\documentclass[12pt]{amsart}

\makeatletter

\theoremstyle{plain}
\newtheorem{thm}{Theorem}[section]
\newtheorem*{thmintro}{Theorem}
\newtheorem{cor}[thm]{Corollary} 
\newtheorem{lem}[thm]{Lemma} 
\newtheorem{prop}[thm]{Proposition} 
\theoremstyle{definition}
\newtheorem{defn}[thm]{Definition}
\theoremstyle{remark}
\newtheorem{rem}[thm]{Remark}
\theoremstyle{remark}

\theoremstyle{remark}

\theoremstyle{remark}

\theoremstyle{definition}

\theoremstyle{definition}

\theoremstyle{plain}

\theoremstyle{definition}

\theoremstyle{remark}

\theoremstyle{remark}

\theoremstyle{definition}

\theoremstyle{remark}
  \newtheorem*{acknowledgement*}{Acknowledgement}

\newcommand{\FF}{{\mathbb F}}
\newcommand{\C}{{\mathbb C}}
\newcommand{\N}{{\mathbb N}}

\newcommand{\M}{{\mathbb M}}

\newcommand{\B}{{\mathbb B}}
\newcommand{\F}{{\mathbb F}}
\newcommand{\G}{\Gamma}
\newcommand{\e}{\varepsilon}
\newcommand{\p}{\varphi}
\newcommand{\hh}{{\mathcal H}}
\newcommand{\hk}{{\mathcal K}}
\newcommand{\id}{\mathrm{id}}

\newcommand{\ip}[1]{\langle#1\rangle}

\def\freeprod{\font\bigsymbolsfont=cmsy10 scaled \magstep3
 \setbox0=\hbox{\bigsymbolsfont\char'003 }\mathord{\lower1pt\box0}}\relax\ignorespaces

\newcommand{\Hawaii}{Hawai\kern.05em`\kern.05em\relax i}
\newcommand{\Manoa}{M\=anoa}

\usepackage{amssymb} \usepackage{amscd}

\makeatother

\begin{document}

\title[New C$^*$-completions of discrete groups]{New C$^*$-completions
  of discrete groups\\ and related spaces} 

\author{Nathanial P.\ Brown and Erik Guentner}

\address{Department of Mathematics, Penn State University, State
College, PA 16802, USA}

\email{nbrown@math.psu.edu} 

\address{Department of Mathematics, University of \Hawaii\ at \Manoa,
  Honolulu, HI 96822}  

\email{erik@math.hawaii.edu} 

\thanks{The first named author was partially supported by
  DMS-0856197. The second named author was partially supported
  by DMS-0349367.} 

\begin{abstract}  
  Let $\Gamma$ be a discrete group.  To every ideal in
  $\ell^{\infty}(\G)$ we associate a C$^*$-algebra completion of the
  group ring that encapsulates the unitary representations with matrix
  coefficients belonging to the ideal.  The general framework we
  develop unifies some classical results and leads to new insights.
  For example, we give the first C$^*$-algebraic characterization of
  a-T-menability; a new characterization of property (T); new examples of ``exotic" quantum groups;
  and, after extending our construction to transformation groupoids, we
  improve and simplify a recent result of Douglas and Nowak \cite{DN}.
\end{abstract}

\maketitle

\section{Introduction}

Since their introduction by von Neumann, amenable groups have played
an important role in many areas of mathematics.  They have been
studied from a variety of perspectives and in many different contexts,
and a vast literature is now devoted to them.  More recently, the
concept of an amenable action of a (non-amenable) group was introduced
by Zimmer, and subsequently developed by many authors.  An elementary
connection between these theories is the fact that every action of an
amenable group is an amenable action.  Less obvious, but equally
well-known, is that if a group acts amenably on a compact space fixing
a probability measure then the group itself is amenable.

This last fact is the launching point of a recent paper by Douglas and
Nowak \cite{DN}, in which, among other things, they introduce
conditions on an amenable action sufficient to guarantee that the
group acting is a-T-menable -- in other words, that it admits a
metrically proper, affine isometric action on a Hilbert space.  An
amenable group is a-T-menable, so that one may imagine hypotheses
involving existence of a quasi-invariant measure together with
conditions on the associated Radon-Nikodym cocycle.  Precisely,
suppose a discrete group $\G$ acts amenably on the compact Hausdorff
topological space $X$, and that $\mu$ is a probability measure on $X$
which is quasi-invariant for the action.  Define upper and lower
envelopes of the Radon-Nikodym cocycle by
\begin{equation*}
  \overline{\rho}(x) = \sup_{s\in G} \frac{ds^*\mu}{d\mu}(x), 
     \quad\text{and}\quad
  \underline{\rho}(x) = \inf_{s\in G} \frac{ds^*\mu}{d\mu}(x); 
\end{equation*}
here, $s^*\mu$ is the translate of the measure $\mu$ by the group
element $s$, and ${ds^*\mu}/{d\mu}$ is the Radon-Nikodym derivative. 
Douglas and Nowak show that if $\overline{\rho}$ is integrable, or if
$\underline{\rho}$ is nonzero, then the group $\G$ is a-T-menable.
They ask whether amenability of $\G$ follows from either of these
conditions.  In this note, we shall prove that this is indeed the
case.  See Corollary~\ref{cor:DN} and surrounding discussion.

Our initial result lead us to the following question: if one wishes to
conclude a-T-menability of $\G$, what are the appropriate hypotheses?
To answer this question, we introduce appropriate completions of the
group ring of $\G$, and of the convolution algebra $C_c(X\rtimes G)$
in the case of an action.  Precisely, for every algebraic ideal in
$\ell^\infty(\G)$ we associate a completion -- for $\ell^\infty(\G)$
we recover the full C$^*$-algebra, for $c_c(\G)$ we recover the
reduced C$^*$-algebra, and for $c_0(\G)$ we obtain new C$^*$-algebras
well-adapted to the study of a-T-menability and a-T-menable actions.
Our results in this context are summarized:

\begin{thmintro}
  Let $\G$ be a discrete group acting on a compact Hausdorff
  topological space $X$.  Let $C^*_{c_0}(\G)$ and 
  $C^*_{c_0}(X\rtimes \G)$ be the completions with respect to the ideal
  $c_0(\G)$.  We have:
  \begin{enumerate}
  \item $\G$ is a-T-menable if and only if $C^*(\G)=C^*_{c_0}(\G)$;
  \item if the action of\/ $\G$ on $X$ is a-T-menable then
            $C^*(X\rtimes \G) = C^*_{c_0}(X\rtimes \G)$.
  \end{enumerate}
  Further, under the hypotheses of Douglas and Nowak, $\G$ is
  a-T-menable if and only if its action on $X$ is a-T-menable.
\end{thmintro}

Apart from this theorem, and ancillary related results, we develop
some general aspects of our ideal completions.  We study when an ideal
completion recovers the full or reduced group C$^*$-algebra and give
examples when it is neither -- this gives rise to `exotic' compact
quantum groups.  We recover a standard characterization of amenability
-- equality of the full and reduced group C$^*$-algebras -- we obtain
the characterization of a-T-menability stated above, and we
characterize Property (T) in terms of ideal completions.

\begin{acknowledgement*} 
The first author thanks the math department at the University of
\Hawaii\ for embodying the aloha spirit during the sabbatical year when
this work was carried out.  He also thanks Yehuda
Shalom and Rufus Willett for helpful remarks and
suggestions, respectively.   Both authors thank Jesse Peterson for
sharing his insights.
\end{acknowledgement*}

\section{Ideals and C$^*$-completions}  

Throughout, $\G$ will denote a (countable) discrete group and $D
\triangleleft \ell^{\infty}(\G)$ will be an \emph{algebraic} (not
necessarily norm-closed) two-sided ideal.  If 
$\pi\colon \G \to \B(\hh)$ is a unitary representation and vectors
$\xi$, $\eta \in\hh$ are given, the $\ell^{\infty}$-function
\begin{equation*}
  \pi_{\xi,\eta}(s) := \ip{\pi_s(\xi), \eta}
\end{equation*}
is a \emph{matrix coefficient \textup{(}function\textup{)}} of $\pi$.
The map associating to a pair of vectors their matrix coefficient
function is sesquilinear; concretely, given finitely many vectors
$v_i$, $w_j\in \hh$, if we set $\xi = \sum \alpha_i v_i$ and 
$\eta = \sum_j \beta_j w_j$ then we have
\begin{equation*}
  \pi_{\xi,\eta} = \sum_{i,j} \alpha_i \bar{\beta}_j \pi_{v_i,w_j}.
\end{equation*}
In particular, if a linear subspace of $\ell^{\infty}(\G)$ contains
the $\pi_{v_i,w_j}$ then it contains $\pi_{\xi,\eta}$ as well.

\begin{defn} 
\label{defn:Drep} 
Let $D\triangleleft \ell^{\infty}(\G)$ be an ideal.  A unitary
representation $\pi\colon \G \to \B(\hh)$ is a \emph{D-representation}
if there exists a dense linear subspace $\hh_0 \subset \hh$ such that
$\pi_{\xi,\eta} \in D$ for all $\xi$, $\eta \in \hh_0$.  
\end{defn}

\begin{defn} 
\label{defn:grpdecay} 
Let $D\triangleleft \ell^{\infty}(\G)$ be an ideal.  Define a
C$^*$-norm on the group ring $\C [\G]$ by 
\begin{equation*}
  \| x \|_D := \sup \{\, \|\pi(x)\| : 
      \pi \textrm{ is a $D$-representation}\};
\end{equation*}
let $C^*_D(\Gamma)$ denote the completion of $\C [\G]$ with
respect to $\| \cdot \|_D$.  
\end{defn} 

We shall refer to the C$^*$-algebra $C^*_D(\Gamma)$, and its
generalizations defined below, as \emph{ideal completions}; these will
be our primary objects of study.  

Evidently, $C^*_D(\G)$ has the universal property that every
$D$-representation of $\G$ extends uniquely to $C^*_D(\G)$.  We shall
refer to such representations as \emph{$D$-representations} of
$C^*_D(\G)$.  

By virtue of its universal property, the full group
C$^*$-algebra of $\G$ surjects onto every ideal completion.  For
some $D$ the ideal completion $C^*_D(\G)$ does not contain the group
ring -- it may even be the zero $C^*$-algebra! -- so is not strictly
speaking a `completion'.  However, if $D$ contains the ideal $c_c(\G)$
of finitely supported functions then $C^*_D(\G)$ is indeed a
completion of the group ring -- this follows because the regular
representation of $\G$, being a $c_c$-representation, extends to
$C^*_D(\G)$.

\begin{rem}
  If $D$ is a \emph{closed} ideal, then \emph{every} matrix
  coefficient of a $D$-representation belongs to $D$ (that is, not
  just those associated to the dense subspace $\hh_0$).  For example,
  this is the case for the ideal $c_0(\G)$ of functions vanishing at
  infinity. 
\end{rem}

\begin{rem}[Tensor products]
The tensor product of a $D$-representation and an arbitrary
representation is again a $D$-representation.  Suppose
$\pi\colon\G \to B(\hh)$ is a $D$-representation and 
$\sigma\colon\G \to B(\hk)$ is arbitrary.   For $v_i\in\hh_0$ and
$w_i\in\hk$ we have
\begin{equation*}
  (\pi \otimes \sigma)_{v_1 \otimes w_1, {v_2} \otimes {w_2}} = 
   \pi_{v_1,{v_2}} \sigma_{w_1,{w_2}} \in D,
\end{equation*}
since $D$ is an ideal; further, such simple tensors have dense span in
$\hh\otimes\hk$. 
\end{rem}

\begin{rem}[Direct sums] 
\label{rem:sum} 
An arbitrary direct sum of $D$-representations is again a
$D$-representation.  This follows since in the definition we only
require a dense subspace.
\end{rem} 

As a consequence, $C^*_D(\G)$ has a \emph{faithful}
$D$-representation. Indeed, for each element $x \in C^*_D(\G)$ there
is a $D$-representation $\pi$ such that $\pi(x) \neq 0$.  Taking
direct sums one easily constructs a faithful $D$-representation of
$C^*_D(\G)$.

\begin{defn} 
\label{defn:D} 
An ideal $D \triangleleft \ell^{\infty}(\G)$ is 
\emph{translation invariant} if it is invariant under both the left
and right translation actions of $\Gamma$ on $\ell^{\infty}(\G)$.   
\end{defn}

Every nonzero, translation invariant ideal in $\ell^\infty(\G)$ contains the
ideal $c_c(\G)$.  It follows
that the ideal completion with respect to a translation invariant
ideal surjects onto the reduced group C$^*$-algebra of $\G$.


\begin{rem}[Cyclic representations]  
\label{rem:cyclic} 
Let $D$ be a translation invariant ideal.  If $v$ is a cyclic vector
for a representation $\pi\colon \G\to\B(\hh)$ and $\pi_{v,v}\in D$,
then $\pi$ is a $D$-representation.  Indeed, a computation confirms
that if $\xi = \pi_{g_1}(v)$ and $\eta = \pi_{g_2}(v)$, then 
\begin{equation*}
  \pi_{\xi,\eta}(s) = \pi_{v,v}(g_2^{-1} s g_1),
\end{equation*}
so that also $\pi_{\xi,\eta} \in D$.  Setting 
$\hh_0 = \textrm{span}\{ \pi_s(v) : s \in \G \}$ we see that $\hh_0$
is dense in $\hh$ and that the matrix coefficients coming from
vectors in $\hh_0$ belong to $D$. 
\end{rem}

\begin{prop}  
  Let $\phi:\Gamma_1\to\Gamma_2$ be a group homomorphism and let
  $D_i\triangleleft\ell^\infty(\Gamma_i)$ be ideals satisfying the
  following condition:  if $f\in D_2$ then $f\circ\phi\in D_1$.  Then
  $\phi$ extends to a \emph{C}$^*$-homomorphism
  $C^*_{D_1}(\Gamma_1)\to C^*_{D_2}(\Gamma_2)$.
\end{prop}
\begin{proof}
  Apply the following simple observation to a faithful
  $D_2$-representation of $\Gamma_2$:  
  under the stated hypotheses, if $\pi$ is a $D_2$-representation of
  $\Gamma_2$, then $\pi\circ\phi$ is a $D_1$-representation of
  $\Gamma_1$; in particular it extends to $C^*_{D_1}(\Gamma_1)$.
\end{proof}

\begin{cor} 
The following assertions hold.  
\label{prop:funct} 
\begin{enumerate} 
\item
\label{prop:funct1} 
Suppose $D_i \triangleleft \ell^{\infty}(\G)$ are ideals and 
$D_2 \subset D_1$; there is a quotient map 
$C^*_{D_1}(\G) \to C^*_{D_2}(\G)$
\textup{(}extending the identity map on the group ring\textup{)}.

\item
\label{prop:funct3} 
Suppose $\Lambda \subset \G$ is a normal subgroup, 
$D_2 \triangleleft \ell^{\infty}(\G / \Lambda)$ is an ideal and 
$D_1 \triangleleft \ell^{\infty}(\G)$ is an ideal containing the image 
of $D_2$ under the inclusion 
$\ell^{\infty}(\G /\Lambda) \subset \ell^{\infty}(\G)$;
there is a surjection 
\textup{(}extending the homomorphism $\phi$\textup{)} 
\begin{equation*}
  C^*_{D_1}(\G) \to C^*_{D_2}(\G / \Lambda).
\end{equation*}
\end{enumerate} 
\end{cor} 

\begin{proof}
  Both statements are immediate from the proposition.  The first is
  also equivalent to the inequality, 
  $\| \cdot \|_{D_1} \leq \| \cdot \|_{D_2}$, which is immediate from
  the definitions. 
\end{proof}

We close this introductory section by looking at several basic
examples.
%
%
Our first example is trivial, since both algebras in question
satisfy the same universal property.  

\begin{prop} 
\label{prop:infty}  For every discrete group $\Gamma$, the completion
with respect to the ideal $\ell^\infty(\Gamma)$ is the universal
\textup{(}or full\textup{)} group \emph{C}$^*$-algebra: 
$C^*_{\ell^\infty}(\Gamma) = C^*(\G)$. \qed
\end{prop}

\begin{prop} 
\label{prop:2} 
For every discrete group $\Gamma$ and every $p \in [1,2]$ the
completion with respect to the ideal $\ell^p(\G)$ is the reduced
group \emph{C}$^*$-algebra:  $C^*_{\ell^p}(\Gamma) = C^*_r(\G)$.
\end{prop} 

\begin{proof} 
This follows from the Cowling-Haagerup-Howe Theorem (cf.\ \cite{CHH}):
if $\pi \colon \G \to B(\hh)$ has a cyclic vector $v\in \hh$ and
$\pi_{v,v} \in \ell^2(\G)$, then $\pi$ is weakly contained in the
regular representation.\footnote{Actually, for the
  Cowling-Haagerup-Howe Theorem it suffices to have
  $\pi_{v,v} \in \ell^{2+\e}(\G)$ for all $\e >0$.  Thus, the
  proposition generalizes to the ideal 
  $D := \cap_{\e > 0} \ell^{2+\e}(\G)$, with exactly the same proof.}
Indeed, fix a nonzero $x  \in C^*_{\ell^p}(\Gamma)$.  We can find a cyclic
$\ell^p(\Gamma)$-representation $\pi$ such that $\pi(x) \neq 0$ -- 
simply restrict a faithful $\ell^p(\Gamma)$-representation to an
appropriate cyclic subspace.   Since  
$\pi$ is weakly contained in the regular representation,
$x$ cannot be in the kernel of the map 
$C^*_{\ell^p}(\Gamma) \to C^*_r(\G)$.  
\end{proof} 

It follows from functoriality (part (\ref{prop:funct1}) of
Corollary~\ref{prop:funct}) that if a translation invariant ideal $D$
is contained in $\ell^p(\Gamma)$ for some $p \in [1,2]$ then
$C^*_D(\G) = C^*_r(\G)$.  This applies, in particular, to the ideal of
finitely supported functions.  In contrast, the ideals $\ell^p(\G)$
for finite $p$ give rise to the universal group C$^*$-algebra only if
$\Gamma$ is amenable.

\begin{prop} 
\label{prop:amenable} 
If there exists  $p \in [1,\infty)$ for which
$C^*(\G) = C^*_{\ell^p}(\G)$, then $\G$ is amenable.   
\end{prop}

In the proof, and at a number of places below, we shall use the notion
of a positive definite function: recall that $h\colon \G \to \C$ is
\emph{positive definite} if for every $s_1, \ldots, s_n \in \G$, the
matrix $[h(s_i s_j^{-1})]_{i,j} \in \M_n(\C)$ is positive
(semidefinite).

\begin{proof}  
If $C^*(\G) = C^*_{\ell^p}(\G)$, then $C^*(\G)$ admits a
\emph{faithful} $\ell^p(\Gamma)$-representation $\pi$ and, taking an 
infinite direct sum if necessary, we may assume $\pi(C^*(\G))$
contains no compact operators.  In this case, Glimm's lemma implies
that $\pi$ weakly contains the trivial representation.  Thus, let
$v_n$ be unit vectors such that $\| \pi_s(v_n) - v_n \| \to 0$ for all 
$s \in \G$.  Approximating the $v_n$'s with vectors having associated
matrix coefficients in $\ell^p(\Gamma)$, we may assume 
$\pi_{v_n, v_n} \in \ell^p(\Gamma)$ for all $n \in \N$.  Since 
$\pi_{v_n, v_n}$ are positive definite functions tending pointwise to
one, we conclude that $\G$ is amenable.
\end{proof}

\begin{rem}
  In the previous proof we have used the following elementary fact: if
  there exist positive definite functions $h_n \in \ell^p(\G)$ for
  which $h_n\to 1$ pointwise, then $\G$ is amenable.  Lacking a
  reference, we provide the following argument.  For $k$ larger than
  $p$ the functions $h_n^k$ are positive definite, converge pointwise
  to one, and belong to $\ell^1(\G)\subset C^*_r(\G)$.  To get
  finitely supported functions with similar properties, consider
  $f_n\in\C(\G)$ which approximate the square roots of the $h_n^k$ in
  the norm of $C^*_r(\G)$, so that $h_n^k$ is approximated by the
  finitely supported positive definite function $f_n^**f_n$.
\end{rem}

\section{Positive definite functions and the Haagerup property} 

Though very simple, the proof of Proposition \ref{prop:amenable}
suggests a general result.  We begin with a lemma isolating the role
of translation invariance.

\begin{lem} 
\label{lem:GNSinvariance}  
Suppose $D\triangleleft \ell^\infty(\G)$ is a translation invariant ideal
and $h \in D$ is positive definite.  The GNS representation
corresponding to $h$ is a $D$-representation.    
\end{lem} 

\begin{proof} Immediate from Remark \ref{rem:cyclic}:
  if $\pi$ is the GNS representation associated to $h$, and $v \in\hh$
  is the canonical cyclic vector, then $\pi_{v,v} = h \in D$.   
\end{proof} 

\begin{thm}  
\label{thm:grpcase}  
Let $D \triangleleft \ell^{\infty}(\G)$ be a translation invariant
ideal.  We have that $C^*(\Gamma) = C^*_D(\G)$ if and only if there
exist positive definite functions $h_n \in D$ converging pointwise to
the constant function $1$.
\end{thm} 

\begin{proof} 
First assume that the canonical map $C^*(\G) \to C^*_D(\G)$ is an
isomorphism.  Replacing $\ell^p(\Gamma)$ with $D$ in the proof of
Proposition \ref{prop:amenable}, we see how to construct the desired
positive definite functions.  

For the converse, suppose $h_n \in D$ are positive definite functions
such that $h_n(s) \to 1$.  To prove that $C^*(\Gamma) = C^*_D(\G)$, it
suffices to observe that vector states coming from $D$-representations
are weak-$*$ dense in the state space of $C^*(\G)$, since this implies
that the map $C^*(\G) \to C^*_D(\G)$ has a trivial kernel.  So
let $\p$ be a  state on $C^*(\G)$.  Then the formula 
\begin{equation*}
  \p_n \left(\sum_{s \in \G} \alpha_s s\right) := 
      \sum_{s \in \G} \alpha_s h_n(s)\p(s)
\end{equation*}
determines a state on $C^*(\G)$  -- it is the composition of $\p$ and
the completely positive Schur multiplier $C^*(\G)\to C^*(\G)$
associated to $h_n$, cf.\ \cite{BO}.  Since the 
norms of the $\p_n$ are uniformly bounded, we have that $\p_n \to
\p$ in the weak-$*$ topology.  Also, it's clear that $\p_n|_{\G} \in
D$ since it is the product of $h_n$ and $\p|_{\G}$.   So the previous
lemma implies the GNS representations associated to the $\p_n$ are
$D$-representations, concluding the proof.  
\end{proof} 

It has been open for some time whether the Haagerup property ($\equiv$
a-T-menability, see \cite{CCJJV}) admits a C$^*$-algebraic
characterization.  The previous theorem easily implies such a
characterization, which is perfectly analogous to a well-known fact
about amenable groups.  To see the parallel, we isolate two more
canonical ideal completions.

\begin{defn}
  Let $C^*_{c_c}(\G)$ denote the ideal completion associated to the
  ideal $c_c(\G)$ of finitely supported functions; let $C^*_{c_0}(\G)$
  denote the ideal completion associated to $c_0(\G)$, the functions
  vanishing at infinitey.
\end{defn}  

Recall that $\G$ is amenable if there exist positive definite
functions in $c_c(\G)$ converging pointwise to one; similarly $\G$ has
the Haagerup property if there exist positive definite functions in
$c_0(\G)$ converging pointwise to one.  Since both $c_c(\G)$ and
$c_0(\G)$ are translation invariant ideals, our next result follows
immediately from Theorem \ref{thm:grpcase}.  Having already observed that
$C^*_{c_c}(\G)=C^*_r(\G)$, the first statement is classical.
The second statement is closely related to
the following fact: $\G$ has the Haagerup property if and
only if it admits a $c_0$-representation weakly containing the trivial
representation \cite{CCJJV}.

\begin{cor} 
\label{cor:grp} 
For a discrete group $\G$ we have:  $\G$ is amenable if and only if
$C^*(\G) = C^*_{c_c}(\G)$; $\G$ has the Haagerup property if and only if
$C^*(\G) = C^*_{c_0}(\G)$. \qed
\end{cor} 

Since $C^*_{c_0}(\G)$ admits a faithful $c_0$-representation, we
have an analogue of the fact that every representation of an amenable group is weakly contained in the left regular representation.

\begin{cor}\label{cor:c_0}
If\/ $\G$ has the Haagerup property, then every unitary representation
is weakly contained in a $c_0$-representation.\footnote{This can also be deduced from the existence of a $c_0$-representation weakly containing the trivial one.}
 \qed
\end{cor}

Jesse Peterson asked if Property (T) can be characterized in this
context, and suggested the following proposition.  For the definition
of Property (T) we refer to \cite{BDV}.

\begin{prop} 
\label{cor:T}
A discrete group $\G$ has Property \emph{(T)} precisely when the
following condition holds:  $\ell^\infty(\G)$ is the only translation
invariant ideal $D$ for which $C^*(\Gamma) = C^*_D(\G)$.
\end{prop} 

\begin{proof}  
  First, suppose that $\G$ has Property (T) and that $D$ is a
  translation invariant ideal for which $C^*_D(\G)=C^*(\G)$.  We must
  show that $D=\ell^\infty(\G)$.  But, by Theorem~\ref{thm:grpcase}
  there exist positive definite functions $h_n \in D$ converging
  pointwise to one.  Since $\G$ has Property (T) they converge
  uniformly to one.  Thus, some $h_n$ is bounded away from zero,
  and so is invertible in $\ell^\infty(D)$.

Conversely, suppose that $\G$ does \emph{not} have Property (T).  Let 
\begin{equation*}
  D = \{\, f\in\ell^\infty(\G) \colon 
              \text{$\inf_{s\notin F} |f(s)| = 0$ 
               for every finite $F\subset\G$} \,\}.
\end{equation*}
One readily checks that $D$ is a proper, translation invariant ideal
in $\ell^\infty(\G)$.  We shall show that $C^*_D(\G)=C^*(\G)$.  By
Theorem~\ref{thm:grpcase} we must exhibit positive definite functions
in $D$ converging pointwise to one.  Since $\G$ does not
have Property (T) there exists an unbounded, conditionally negative
type function $\psi$ on $\G$; the desired functions are 
$h_n = e^{-\psi/n}$.
\end{proof}

\section{Quantum groups}

Recall that a \emph{compact quantum group} is a pair $(A, \Delta)$
where $A$ is a unital C$^*$-algebra and 
$\Delta \colon A \to A \otimes A$ is a unital $*$-homomorphism
satisfying the following two conditions:

\begin{enumerate} 
\item $(\Delta \otimes \id_A) \Delta = (\id_A \otimes \Delta) \Delta$, and

\item $\Delta(A)(A \otimes 1)$ and $\Delta(A)(1 \otimes A)$ are dense
  subspaces of $A\otimes A$.\footnote{All tensor products in this
    definition are spatial (cf.\ \cite{BO}).}  
\end{enumerate} 
The map $\Delta$ is the \emph{co-multiplication} and the first
property is called \emph{co-associativity}.

Discrete groups provided an early source of examples of quantum
groups.  Indeed, the assignment $\Delta(s) =  s \otimes s$ on group
elements determines a co-associative map
\begin{equation}
\label{delta}
  \Delta \colon \C[\G] \to \C[\G] \otimes \C[\G],
\end{equation}
and one can check that $\Delta(\C[\G])(\C[\G] \otimes 1) =
\Delta(\C[\G])(1 \otimes \C[\G]) = \C[\G]\otimes \C[\G]$.  Thus, if
$A$ is a C$^*$-algebra containing $\C[\G]$ as a dense subalgebra and
for which $\Delta$ can be extended continuously to a map $A \to A
\otimes A$, then $A$ is a compact quantum group.  Every discrete group
gives rise to two canonical compact quantum groups: the universal
property ensures that $A=C^*(\G)$ is a compact quantum group whereas
Fell's absorption principle implies that $A=C_r^*(\G)$ is a compact
quantum group.  Anything between these extremes is considered `exotic'
(cf.\ \cite{KS}).  The purpose of this section is to provide examples
of such exotic compact quantum groups.

\begin{prop}  
For every ideal $D \triangleleft \ell^{\infty}(\G)$, the ideal
completion $C^*_D(\G)$ is a
compact quantum group.  
\end{prop} 

\begin{proof} 
We shall show that the map $\Delta$ in (\ref{delta}) extends
continuously to $C^*_D(\G) \to C^*_D(\G) \otimes C^*_D(\G)$.  
To this end, fix a faithful $D$-representation $C^*_D(\G)
\subset B(\hh)$.  By Remark~\ref{prop:funct}, we can regard 
the composite
\begin{equation*}
  \Delta \colon \C[\G] \to C^*_D(\G)\otimes C^*_D(\G) 
        \subset B(\hh\otimes \hh)
\end{equation*}
as a $D$-representation and the universal property ensures that
$\Delta$ extends to $C^*_D(\G)$. 
\end{proof}

Thus our task is to provide examples of groups $\G$ and ideals 
$D \triangleleft \ell^{\infty}(\G)$ for which 
$C^*(\G) \neq C^*_D(\G) \neq C^*_r(\G)$.  Though a bit ad hoc, our
first examples are easy to handle.

\begin{prop} 
\label{prop:exoticT} 
Suppose that $\G$ has Property \textup{(}T\textup{)} and that
$\Lambda\triangleleft\G$ is a non-amenable normal subgroup of
infinite index.  Suppose $D_1\triangleleft\ell^\infty(\G)$ and
$D_2\triangleleft\ell^\infty(\G/\Lambda)$ are ideals satisfying the
following conditions:
\begin{enumerate}
\item $D_1$ is proper and translation invariant;
\item $D_2$ is translation invariant;
\item $D_1$ contains the image of $D_2$ under the inclusion
  $\ell^\infty(\G/\Lambda)\subset \ell^\infty(\G)$.
\end{enumerate}
Then $C^*_{D_1}(\G)$ is an exotic compact quantum group.
\end{prop}

\begin{proof} 
We must show that $C^*(\G)\neq C^*_{D_1}(\G)\neq C^*_r(\G)$.  
Since $D_1$ is proper and translation invariant, and $\G$ has Property
(T) the first inequality follows from Proposition~\ref{cor:T}.

To prove the second inequality, suppose to the contrary that
$C^*_{D_1}(\G)=C^*_r(\G)$. Applying (\ref{prop:funct3}) of
Corollary~\ref{prop:funct} we obtain a $*$-homomorphism 
\begin{equation*}
  C^*_r(\G) = C^*_{D_1}(\G) \to C^*_{D_2}(\G/\Lambda) 
      \to C^*_r(\G/\Lambda)
\end{equation*}
extending the homomorphism $\G\to \G/\Lambda$.  It follows that
$\Lambda$ is amenable -- the composite 
$C^*_r(\Lambda) \subset C^*_r(\G) \to C^*_r(\G/\Lambda)$ defines a
character of $C^*_r(\Lambda)$.
\end{proof} 

\begin{rem} 
While the hypotheses of the previous proposition may seem a bit
contrived, examples are plentiful.  An extension of Property (T)
groups will again have Property (T)
\cite[Proposition 1.7.6]{BDV}. As for the ideals, taking $D_1$ to
be the ideal generated by the image of $D_2$ in $\ell^\infty(\G)$ we
have: if $D_2$ is translation invariant, then so is $D_1$; if 
$D_2 \subset c_0(\G/\Lambda)$, then $D_1$ is proper.  
\end{rem} 

Free groups also provide natural examples. We thank Rufus
Willett for suggesting the proof of the following result.

\begin{prop}
\label{prop:freegroup} 
Let $\FF$ be a free group on two or more generators.  There
exists a $p \in (2,\infty)$ such that 
$C^*(\FF) \neq C^*_{\ell^p}(\FF) \neq C^*_r(\FF)$.  
\end{prop} 

\begin{proof} 
Since $\FF$ is not amenable, Proposition~\ref{prop:amenable} implies
that $C^*(\FF) \neq C^*_{\ell^p}(\FF)$ for all finite $p$.  We must find
some $p$ such that $C^*_{\ell^p}(\FF) \neq C^*_r(\FF)$. 

Let $S \subset \FF$ be the standard symmetric generating set and let $| \cdot |$
denote the corresponding word length.  A seminal result, first proved
by Haagerup \cite{H}, states that for every $n \in\mathbb{N}$, 
\begin{equation*}
  h_n(s) := e^{-| s |/n}
\end{equation*}
is positive definite.  Clearly $h_n \to 1$ pointwise.  Fixing $n$, we
have  $h_n \in \ell^{p_n}(\Gamma)$ for sufficiently large $p_n$;
indeed if $p_n$ is chosen so that  
$|S| < e^{p_n/n}$, or equivalently 
$|S| e^{-p_n/n} < 1$, then
\begin{equation*}
  \sum_{s \in \G} (e^{-|s|/n})^{p_n} = 
\sum_{k = 1}^{\infty} \bigg( \sum_{|s| = k} e^{-kp_n/n} \bigg) 
\leq \sum_{r = 1}^{\infty} \big( |S|^k e^{-kp_n/n} \big) 
= \sum_{r = 1}^{\infty} \big( |S| e^{-p_n/n} \big)^k 
< \infty.
\end{equation*}

Let $\pi_n \colon C^*_{\ell^{p_n}} (\G) \to B(\hh_n)$ be the GNS
representations corresponding to $h_n$, and let $v_n \in \hh_n$ be
the canonical cyclic vector.  Since since $h_n(s) \to 1$ we see that 
$\| \pi_n(s)v_n - v_n \| \to 0$, for all $s \in \G$.  Hence the direct
sum representation $\oplus \pi_n$ weakly contains
the trivial representation.  It follows that we cannot have
$C^*_{\ell^{p_n}}(\G) = C^*_r(\G)$ for all $n$ -- otherwise
$\oplus \pi_n$ would be defined on $C^*_r(\G)$
and nonamenability prevents the trivial representation from being
weakly contained in any representation of $C^*_r(\G)$.  
\end{proof} 

\begin{rem}
  The previous proposition is not optimal; Higson, Ozawa and Okayasu
  \cite{O} have independently shown that the $C^*$-algebras
  $C^*_{\ell^p}(\F_n)$ are mutually non-isomorphic.  On the other
  hand, 
  extracting the crucial ingredients from the proof, we see
  that the phenomenon presented there is very general.  Indeed,
  suppose that $\G$ is a non-amenable, a-T-menable group admitting an $\N$-valued
  conditionally negative type function $\psi$ satisfying an estimate
  of the following form: there exists $C>0$ such that for every $k$ we
  have
\begin{equation}
\label{psigrowth}
  \# \{\, s\in\G \colon \psi(s) = k \,\} \leq C^k.
\end{equation}
Taking $h_n(s) = e^{-\psi(s)/n}$ the above proof applies verbatim to
show that $C^*_{\ell^p}(\G)\neq C^*_r(\G)$ for some $p$.  This
applies, for example, to infinite Coxeter groups -- the word length
function corresponding to the standard Coxeter generators satisfies
the hypothesis for $\psi$ \cite{BJS}.
\end{rem} 

\begin{rem}
  Continuing the previous remark, suppose a non-amenable group
  $\Gamma$ acts (cellularly) on a $\hbox{CAT}(0)$ cube complex $X$.
  The combinatorial distance $d$ in the one skeleton of $X$ defines an
  $\N$-valued conditionally negative type function on $\Gamma$ by
  \begin{equation*}
    \psi(x) = d(x_0,s\cdot x_0),
  \end{equation*}
  where $x_0$ is any arbitrarily chosen vertex in $X$
  \cite{NR}.\footnote{While stated only for finite dimensional
    complexes, the proof given is valid in greater generality.}   If
  $\Gamma$ is finitely generated and the orbit map $s\mapsto s\cdot
  x_0: \Gamma\to X$ is a quasi-isometric embedding then the inequality
  (\ref{psigrowth}) is satisfied.  These two conditions hold
  in many common situations: by the Svarc-Milnor Lemma, they are
  automatic if the action is proper and cocompact
  \cite{BH} and the complex is finite dimensional; they also hold 
  for the action of Thompson's group $F$
  or, more generally, a finitely generated
  diagram group with {\it Property B\/}, on its Farley complex
  \cite{AGS}.
\end{rem}



\section{Topological Dynamical Systems}

Let $\G$ be a discrete group, and let $X$ be a compact Hausdorff space
on which $\G$ acts by homeomorphisms.  Thinking of transformation
groupoids, let $C_c (X \rtimes \G)$ denote the convolution algebra of
compactly supported functions on $X \times \G$.  We shall represent
elements of this algebra as finite formal sums $\sum f_s s$, where
each $f_s\in C(X)$; we shall view $\G$ as a subset of the convolution
algebra in the obvious manner.

\begin{defn} 
Let $D \triangleleft \ell^{\infty}(\G)$ be an ideal.   A
$*$-representation $\pi\colon C_c (X \rtimes \G) \to \B(\hh)$ is a
\emph{$D$-representation} if $\pi|_\G$ is a $D$-representation in the
sense of Definition~\ref{defn:Drep}.  
\end{defn} 

\begin{defn} 
\label{defn:dynamdecay} 
Let $D \triangleleft \ell^{\infty}(\G)$ be an ideal.  
Define a C$^*$-norm $\| \cdot \|_D$ on $C_c (X \rtimes \G)$ by 
\begin{equation*}
  \left\| \sum f_s s \right\|_D := 
         \sup \left\{\, \left\|\pi\left(\sum f_s s\right)\right\| : 
            \pi \textrm{ is a D-representation} \,\right\},
\end{equation*}
and let $C^*_D(X \rtimes \G)$ denote the completion of $C_c (X \rtimes
\G)$ with respect to $\| \cdot \|_D$.  
\end{defn} 

As before, every $D$-representation extends uniquely to 
$C^*_D(X \rtimes \G)$; further, $C^*_D(X \rtimes \G)$ admits a
\emph{faithful} $D$-representation.  
Considering the universal properties, one sees that 
$C^*_{\ell^\infty}(X \rtimes \G)$ is the universal
(or full) crossed product C$^*$-algebra, denoted $C^*(X\rtimes\G)$.
It is not clear whether the 
analogue of Proposition \ref{prop:2} holds in the present context.

Recall that a function $h\colon X\times \G \to \C$ is \emph{positive
  definite} if for each finite set $s_1, \ldots, s_n \in \G$ and point
$x \in X$, the matrix 
\begin{equation*}
  [h(s_i .x, s_i s_j^{-1})]_{i,j} \in \M_n(\C)
\end{equation*}
is positive (semi-definite); here $x \mapsto s.x$ denotes the action
of $s$ on $X$.  Recall also that to each positive definite $h$ we can
associate a completely positive Schur multiplier
\begin{equation*}
  \textrm{m}_{h} \colon C^*(X \rtimes \G) \to C^*(X \rtimes \G);
\end{equation*}
on finite sums $\textrm{m}_h$ is given by the formula
\begin{equation*}
  \textrm{m}_{h}\left(\sum f_s s\right) = 
            \sum f_s h( s) s
\end{equation*}
where, slightly abusing notation, we have written $h(s)\in C(X)$ for
the function $x \mapsto h(x,s)$ \cite[Proposition~5.6.16]{BO}.

\begin{lem} 
\label{lem:pdgns} 
Suppose that $D \triangleleft \ell^{\infty}(\G)$ is a translation
invariant ideal; suppose that $h\colon X\times \G \to \C$ is
positive definite, and that the function $H(s) : = \| h(s) \|$ belongs 
to $D$. Then, for every state $\p$ on $C^*(X \rtimes \G)$, the GNS
representation associated to $\p \circ \textrm{m}_{h}$ is a
$D$-representation.  
\end{lem}

\begin{proof}  
Note that for $f \in \ell^{\infty}(\G)$ we have, 
$f \in D \Longleftrightarrow |f| \in D$ -- this follows from the polar
decomposition $f = u|f|$, in which $u \in \ell^{\infty}(\G)$ is the
unitary of `pointwise rotation'.  Also, $D$ is hereditary in the sense that if $0 \leq g \leq f$ and $f \in D$, then $g \in D$.  To see this, define $h \in \ell^\infty (\G)$ to be zero wherever $f$ is, and $h(s) = \frac{g(s)}{f(s)}$ otherwise; evidently $g = hf \in D$. 

Now, fix two functions $f$, $g \in C(X)$ and two group elements $s$,
$t \in \G$.  Let $v$ denote the canonical image of $f s$ in the GNS
Hilbert space; similarly let $w$ denote the image of $gt$.  Since the
linear span of such elements is dense, it suffices to show 
$\pi_{v,w} \in D$, where $\pi$ denotes the GNS representation. By the
first paragraph, and our assumptions on $D$ and $H$, it suffices to
show $|\pi_{v,w} |$ is bounded above by a constant times some
translate of $H$.  This, however, is a straightforward calculation.
\end{proof} 

With the previous lemma in hand, the proof of the following result is
very similar to its analog in the group case,
Theorem~\ref{thm:grpcase} -- use Schur multipliers to approximate
arbitrary states by vector states.  Whereas Theorem~\ref{thm:grpcase}
was an `if-and-only-if' statement, we do not know if the converse
holds.

\begin{thm}  
\label{thm:dynamcase}

Let $D \triangleleft \ell^{\infty}(\G)$ be a translation
invariant ideal.  Assume there exist positive definite functions 
$h_n \colon X \times \G \to \C$ satisfying $h_n \to 1$ uniformly on
compact sets and for which each $H_n(s) := \|h_n (s)\|$ belongs to
$D$.  Then $C^*(X \rtimes \G) = C^*_D(X \rtimes \G)$. \qed
\end{thm} 


\begin{defn}
\label{defn:a-T-menable}  
An action of $\G$ on $X$ is \emph{amenable} if there exist positive
definite functions $h_n \in C_c(X \rtimes \G)$ such that $h_n \to 1$
uniformly on compact sets; it is \emph{a-T-menable} if there exist
positive definite functions $h_n \in C_0(X \rtimes \G)$ such that 
$h_n \to 1$ uniformly on compact sets.
\end{defn} 

\begin{rem}\label{rem:a-T-action}
Just as every action of an amenable group is amenable, every action of
an a-T-menable group is a-T-menable.  Indeed, if $h \in c_0(\G)$ is
positive definite, then a computation confirms that the function  
$\tilde{h} \in C_0(X \times \G)$ defined by
$\tilde{h}(x,s) = h(s)$ is as well.  The assertion now follows easily
from the definitions.   
\end{rem} 

\begin{defn}  
  Let $C^*_{c_c}(X\rtimes \G)$ denote the ideal completion associated to
  the ideal of finitely supported functions on $\G$; let
  $C^*_{c_0}(X\rtimes \G)$ denote the ideal completion associated to the
  ideal of functions vanishing at infinity.
\end{defn} 

We draw several corollaries, the analogs of of
Corollaries~\ref{cor:grp} and \ref{cor:c_0} in the group case.  Again,
whereas Corollary~\ref{cor:grp} was an equivalence, the converse of
the first corollary is open.

\begin{cor} 
\label{cor:action} 
Let $\G$ be a discrete group, acting on $X$.  If the action is
amenable then $C^*(X \rtimes \G) = C^*_{c_c}(X \rtimes\G)$; if the
action is a-T-menable, then 
$C^*(X \rtimes \G) = C^*_{c_0}(X \rtimes \G)$.  \qed
\end{cor}

\begin{cor}
\label{cor:weakcont} 
Let $\G$ be a discrete group, acting on $X$.  If the action is
amenable then every covariant representation \textup{(}that is,
$*$-homomorphism $C^*(X\rtimes \G) \to \B(\hh)$\textup{)} is weakly 
contained in a $c_c$-representation; if the action is a-T-menable
then every covariant representation is weakly contained in a
$c_0$-representation.
\end{cor} 

\begin{proof} 
In the case of an amenable action, the hypothesis implies that 
$C^*(X \rtimes \G) = C^*_{c_c}(X \rtimes \G)$, and
$C^*_{c_c}(X \rtimes \G)$ has a faithful $c_c$-representation.  For
an a-T-menable action systematically replace $c_c(\G)$ by $c_0(\G)$
throughout. 
\end{proof} 

\begin{cor} 
  Let $\G$ be a discrete group acting on $X$.  Let  
  $\pi \colon C^*(X\rtimes \G) \to \B(\hh)$ be a covariant
  representation for which $\pi|_\G$ weakly contains the trivial
  representation of $\G$.  We have:
\begin{enumerate} 
\item the action is amenable if and only if $\G$ is amenable; 
\item the action is a-T-menable if and only if $\G$ is a-T-menable.  
\end{enumerate} 
\end{cor} 

\begin{proof}  
The `if' statements are trivial (see Remark~\ref{rem:a-T-action}).
For the `only if' statements, let $D$ stand for the appropriate ideal,
either $c_c(\G)$ or $c_0(\G)$.  Corollary~\ref{cor:action} implies
that $C^*_D(X\rtimes \G) = C^*(X\rtimes \G)$, so that we can form the
composition
\begin{equation*}
C^*_D(\G) \to C^*_D(X\rtimes \G) = C^*(X\rtimes \G) \to \B(\hh).
\end{equation*}
The proof is completed by 
recalling that $\G$ is amenable (respectively,
  a-T-menable) if and only if there is a $c_c(\G)$ (respectively,
  $c_0(\G)$)-representation of $\G$ weakly containing the trivial
  representation. 
\end{proof}  

To close, we return to the result of Douglas and Nowak described in
the introduction.  Suppose a
discrete group $\G$ acts on $X$, and that $\mu$ is a quasi-invariant
measure on $X$ (in other words, elements of $\G$ map $\mu$-null sets
to $\mu$-null sets).  For each element $s\in\G$ the Radon-Nikodym
derivative is the non-negative, measurable function $\rho_s$
satisfying 
\begin{equation*}
  \int_X f d\mu = \int_X s.f \rho_s d\mu,
\end{equation*}
for every measurable function $f$; here, $f \mapsto s.f$ denotes the
action of $s$ on $f$.   The $\rho_s$ allow one to
construct a covariant representation of $X\rtimes\G$ on the Hilbert
space $L^2(X,\mu)$:  functions in $C(X)$ act by multiplication
and elements $s\in\G$ act by the unitaries 
\begin{equation*}
  U_s(f) := (s.f) \rho^{1/2}.
\end{equation*}

\begin{cor} 
\label{cor:DN} 
Let $\G$ be a discrete group acting on $X$, with quasi-invariant
measure $\mu$.  Suppose the representation of\/ $\G$ on $L^2(X,\mu)$
weakly contains the trivial representation.   We have:
\begin{enumerate}
\item the action is amenable if and only if\/ $\G$ is amenable;
\item the action is a-T-menable if and only if\/ $\G$ is a-T-menable.
\end{enumerate}
In particular, such an action of a non-a-T-menable group can never be
a-T-menable.   \qed
\end{cor}

This result, which follows immediately from the previous corollary,
generalizes Theorem~3 and Corollary~4 of Douglas and Nowak 
\cite{DN} -- the hypotheses of their results imply the existence of a
non-zero fixed vector in $L^2(X,\mu)$, namely the square root of
$\overline{\rho}$ or $\underline{\rho}$, as appropriate. See
\cite[Lemma 6]{DN}.

For invariant probability measures we get an analogue of a well-known
amenability result.  

\begin{cor} 
Let $\G$ be a discrete group acting on $X$, with invariant probability measure 
  $\mu$. The action is a-T-menable if and only if\/ $\G$ is
  a-T-menable. In particular, no measure-preserving action of a
  non-a-T-menable group can be a-T-menable.  
\end{cor} 

\begin{proof} 
The constant functions in $L^2(X,\mu)$ are invariant for $\G$.
\end{proof}


\end{document}